\documentclass{article}
\usepackage{amsmath,amsfonts,amsthm,amssymb,amscd,color,xcolor}
\usepackage{graphicx,amscd,mathrsfs,eufrak,wrapfig,mathrsfs,lipsum}
\usepackage{microtype}
\usepackage{float}
\usepackage{tikz}
\usepackage{multicol}
\usepackage{caption}
\usetikzlibrary{arrows}
\usepackage{capt-of}
\usepackage{hyperref}

\colorlet{darkblue}{blue!50!black}

\hypersetup{
    colorlinks,%
    citecolor=darkblue,%
    filecolor=red,%
    linkcolor=darkblue,%
    urlcolor=magenta,%
    pdfnewwindow=true,%
    pdfstartview={FitH}
}

\newcommand{\nnorm}[2]{|\!|\!|#1 |\!|\!|_{#2}}

\newcommand{\esssup}{\mathop{\rm ess\ sup}}

\newcommand{\p}{\partial}
\newcommand{\e}{\varepsilon}

\newcommand{\R}{{\mathbb R}}

\newcommand{\CC}{{\cal C}}

\newcommand{\EE}{{\cal E}}
\newcommand{\FF}{{\cal F}}

\newcommand{\KK}{{\cal K}}
\newcommand{\LL}{{\cal L}}

\newcommand{\RR}{{\cal R}}
\newcommand{\sS}{{\cal S}}

\newcommand{\XX}{{\cal X}}

\newcommand{\fff}{{\mathbf f}}

\theoremstyle{plain}
\newtheorem{theorem}{Theorem}[section]

\newtheorem{lemma}[theorem]{Lemma}
\newtheorem{proposition}[theorem]{Proposition}

\theoremstyle{definition}
\newtheorem{definition}[theorem]{Definition}
\newtheorem*{definition*}{Definition}

\theoremstyle{remark}
\newtheorem{remark}[theorem]{Remark}

\RequirePackage{mathpazo}

\numberwithin{equation}{section}

\begin{document}
\author{Armen Shirikyan} 
\title{Control theory for the Burgers equation: Agrachev--Sarychev approach}
\date{\small Department of Mathematics, University of Cergy--Pontoise, 
CNRS UMR 8088\\ 
2 avenue Adolphe Chauvin, 95302 Cergy--Pontoise, France\\[4pt]
Centre de Recherches Math\'ematiques, CNRS UMI 3457\\
Universit\'e de Montr\'eal, Montr\'eal,  QC, H3C 3J7, Canada\\[4pt]
E-mail: \href{mailto:Armen.Shirikyan@u-cergy.fr}{Armen.Shirikyan@u-cergy.fr}}
\maketitle

\hfill{\sl To Professor Viorel Barbu, on the occasion of his $75^\text{th}$ birthday}

\bigskip
\begin{abstract}
This paper is devoted to a description of a general approach introduced by Agra\-chev and Sarychev in~2005 for studying some control problems for Navier--Stokes equations. The example of a 1D Burgers equation is used to illustrate the main ideas. We begin with a short discussion of the Cauchy problem and establish a continuity property for the resolving operator. We next turn to the property of approximate controllability and prove that it can be achieved by a two-dimensional external force. Finally, we investigate a stronger property, when the approximate controllability and the exact controllability of finite-dimensional functionals are proved simultaneously.


\end{abstract}
\tableofcontents
\setcounter{section}{-1}

\section{Introduction}
\label{s0}
In the paper~\cite{AS-2005}, Agrachev and Sarychev introduced a new approach for investigating the controllability of nonlinear PDEs. They studied the 2D Navier--Stokes equations on a torus controlled by a finite-dimensional external force and proved the properties of approximate controllability and exact controllability in finite-dimensional  projections. These results were later extended to the Euler and Navier--Stokes systems on various 2D manifolds; see~\cite{AS-2006,rodrigues-2006,AS-2008}. 

The Agrachev--Sarychev approach was developed in many works, and similar controllability results were established  for a number of nonlinear PDEs, including some equations for which the well-posedness of the Cauchy problem is not known to hold. Namely, the 3D Navier--Stokes equations were studied in~\cite{shirikyan-CMP2006,shirikyan-AIHP2007}, Nersisyan~\cite{nersisyan-2010,nersisyan-2011} investigated the 3D incompressible and compressible Euler systems, and Sarychev~\cite{sarychev-2012} studied the cubic Schr\"o\-dinger equation on a  2D torus. The Lagrangian (approximate) controllability of the 3D Navier--Stokes equations was proved by Nersesyan~\cite{nersesyan-2015}, and the approximate controllability of the 1D Burgers equation with no decay condition at infinity was established in~\cite{shirikyan-2014}. 

Let us mention that there is enormous literature on the problem of controllability for nonlinear PDEs (e.g., see the books~\cite{fursikov2000,coron2007,BC2016} and the references therein). However, we do not discuss those works here, since our main focus is the Agrachev--Sarychev approach. We shall give a concise self-contained account of their method, using the example of  the 1D Burgers equation
\begin{equation} \label{2.1}
\p_t u-\nu\p_x^2u+u\,\p_xu=h(t,x)+\eta(t,x), \quad x\in (0,\pi),
\end{equation}
where $\nu>0$ is a fixed parameter, $h$ is a given function, and~$\eta$ is a control. Equation~\eqref{2.1} is supplemented with the Dirichlet boundary condition and an initial condition at $t=0$. It will be proved that, given any $L^2$ function~$\hat u$ and a continuous mapping $F:L^2\to\R^N$ that possesses a right inverse on a ball centred at~$F(\hat u)$, any initial point can be steered to an arbitrary small neighbourhood of~$\hat u$ in such a way that the value of~$F$ on the solution coincides with~$F(\hat u)$; see Section~\ref{s3} for the exact formulation. Finally, let us emphasise that the goal of this  paper is to illustrate the Agrachev--Sarychev method on a simple example, and we do not aim at doing it under the most general hypotheses; the results presented in this paper can certainly be extended in many directions. 

\smallskip
The paper is organised as follows. In Section~\ref{s1}, we recall a well-posedness result for the Burgers equation and establish some estimates and continuity properties for the resolving operator. Section~\ref{s2} is devoted to the problem of approximate controllability. We formulate the result and give its detailed proof. In Section~\ref{s3}, we establish the main result of the paper, extending the property of approximate controllability. The appendix gathers some auxiliary assertions used in the main text. 

\smallskip
{\bf Acknowledgements.} 
This paper is an extended version of the lectures on the control theory delivered at the University of Ia\c si in 2010, and it is my pleasure to thank V.~Barbu, T.~Hav\^arneanu, and C.~Popa for invitation and excellent working conditions. This research was carried out within the MME-DII Center of Excellence (ANR-11-LABX-0023-01) and supported by {\it Initiative d'excellence Paris-Seine\/}, the CNRS PICS {\it Fluctuation theorems in stochastic systems\/}, and {\it Agence Nationale de la Recherche\/} through  the grant ANR-17-CE40-0006-02.
 
\subsection*{Notation}
We write $I=[0,\pi]$ and $J_t=[0,t]$ for $t>0$. For a closed interval $J\subset\R$ and a Banach space~$X$, we shall use the following functional spaces.

\medskip
\noindent
$L^2=L^2(I)$ is the space of square-integrable measurable functions $u:I\to\R$; the corresponding norm and inner product are denoted by~$\|\cdot\|$ and~$(\cdot,\cdot)$.

\smallskip
\noindent
$H^s=H^s(I)$ denotes the Sobolev space of order~$s$ on the interval~$I$ with the standard norm~$\|\cdot\|_s$.

\smallskip
\noindent
$H_0^s=H_0^s(I)$ stands for the closure in~$H^s$ of the space of infinitely smooth functions with compact support. 

\smallskip
\noindent
$C(J,X)$ denotes the space of bounded continuous functions $u:J\to X$.

\smallskip
\noindent
$L^p(J,X)$ is the space of Borel-measurable functions $u:J\to X$ such that 
$$
\|u\|_{L^p(J,X)}=\biggl(\int_J\|u(t)\|_X^pdt\biggr)^{1/p}<\infty\,;
$$
in the case $p=\infty$, this norm should be replaced by $\|u\|_{L^\infty(J,X)}=\esssup_{t\in J}\|u(t)\|_X$.

\smallskip
\noindent
We denote $\XX(J)=C(J,L^2)\cap L^2(J,H_0^1)$. In the case $J=J_T$, we shall write~$\XX_T$.

\smallskip
\noindent
$\LL(X,Y)$ is the space of continuous linear operator from~$X$ to~$Y$. 

\section{Cauchy problem}
\label{s1}
\subsection{Well-posedness}
\label{s1.1}
Let us consider the Burgers equation on the interval $I=[0,\pi]$ with the Dirichlet boundary condition:
\begin{align}
\p_t u-\nu\p_x^2u+u\p_x u&=f(t,x), \label{1.1}\\
u(t,0)=u(t,\pi)&=0. \label{1.2}
\end{align}
Here $u=u(t,x)$ is a real-valued unknown function, $\nu>0$ is a parameter, and~$f$ is a given function. Equations~\eqref{1.1}, \eqref{1.2} are supplemented with the initial condition
\begin{equation} \label{1.3}
u(0,x)=u_0(x). 
\end{equation}
The following theorem establishes the well-posedness of the Cauchy problem for the Burgers equation in an appropriate functional space. 

\begin{theorem} \label{t1.1}
Let $T$ and~$\nu$ be some positive numbers. Then, for any $u_0\in L^2$ and $f\in L^1(J_T,L^2)$, there is a unique function $u\in \XX_T$ that satisfies \eqref{1.1}--\eqref{1.3}. 
\end{theorem}

\begin{proof}
We confine ourselves to a formal derivation of an a priori estimate for solutions and to the proof of uniqueness of solution. A detailed account of initial--boundary value problems for some non-linear PDEs can be found in~\cite{lions1969,taylor1996}. 

{\it A priori estimate\/}.
Let us set 
$$
\EE_u(t)=\|u(t)\|^2+2\nu\int_0^t\|\p_xu(s)\|^2ds. 
$$
We multiply Eq.~\eqref{1.1} by~$2u$ and integrate over $I\times J_r$. After some simple transformations, we get
\begin{align*}
\EE_u(r)&=\|u_0\|^2+2\int_0^r\bigl(f(s),u(s)\bigr)\,ds\\
&\le \|u_0\|^2+2\,\|f\|_{L^1(J_r,L^2)}\Bigl(\,\sup_{0\le s\le r}\|u(s)\|\Bigr).
\end{align*}
Taking the supremum over $r\in[0,t]$, we see that
\begin{equation} \label{1.4}
\EE_u(t)\le 2\|u_0\|^2+4\,\|f\|_{L^1(J_t,L^2)}^2\quad\mbox{for $0\le t\le T$}. 
\end{equation}

\smallskip
{\it Uniqueness\/}. If $u_1,u_2\in\XX_T$ are two solutions, then the difference $u=u_1-u_2$ satisfies the equation
$$
\p_t u-\nu\p_x^2u+u\p_x u_1+u_2\p_xu=0.
$$
Multiplying this equation by~$2u$, integrating over $I\times J_t$, and using the relations 
\begin{align*}
	(u_2\p_xu,2u)=-(\p_xu_2,u^2),\quad
	\|u^2\|\le \|u\|_{L^\infty}\|u\|\le C\,\|u\|_{H^1}\|u\|, 
\end{align*}
we derive
\begin{align*}
\EE_u(t)&=\iint_{I\times J_t}u^2(\p_xu_2-2\p_xu_1)\,dxds\\
&\le \int_0^tg(s)\|u(s)\|_{H^1}\|u(s)\|ds\\
&\le \|u\|_{L^2(J_t,H^1)}\biggl(\int_0^tg^2(s)\|u(s)\|^2ds\biggr)^{1/2},
\end{align*}
where $g(t)=C\,\|\p_xu_2-2\p_xu_1\|$ is an $L^2$ function of time. Estimating $\|u(t)\|$ and $\|u\|_{L^2(J_t,H^1)}$ by $\sqrt{\EE_u(t)}$, it follows that 
$$
\EE_u(t)\le (2\nu)^{-1}\int_0^tg^2(s)\EE_u(s)\,ds.
$$
Applying the Gronwall inequality, we conclude that $u\equiv0$. 
\end{proof}

\begin{remark} \label{r1.2}
Let us denote by~$\RR:L^2\times L^1(J_T,L^2)\to \XX_T$ the resolving operator for problem \eqref{1.1}--\eqref{1.3}, that is, a non-linear mapping that takes a pair~$(u_0,f)$ to the solution $u\in\XX_T$. Using rather standard techniques (e.g., see the book~\cite{taylor1996} and the references therein), one can prove that~$\RR$ is uniformly Lipschitz continuous on bounded subsets. Moreover, the same property is true  when $L^1(J_T,L^2)$ is replaced by $L^2(J_T,H^{-1})$.
\end{remark}

\begin{remark} \label{r1.3}
The above-mentioned results are valid in a slightly more general setting. Namely, let us consider the equation 
\begin{equation} \label{1.0}
\p_t u-\nu\p_x^2(u+w)+(u+v)\,\p_x(u+v)=f(t,x), \quad x\in (0,\pi), 
\end{equation}
supplemented with the initial--boundary conditions~\eqref{1.2} and~\eqref{1.3}. One can prove that, for any $u_0\in L^2$ and any functions
$$
v\in\XX_T+L^2(J_T,H^2), \quad w\in L^1(J_T,H^2),\quad
f\in L^1(J_T,L^2)+L^2(J_T,H^{-1}),
$$  
problem~\eqref{1.0}, \eqref{1.2}, \eqref{1.3} has a unique solution $u\in\XX_T$, and the associated resolving operator that takes $(v,w,f,u_0)$ to~$u$ is uniformly Lipschitz continuous on bounded subsets. 
\end{remark}

In what follows, we denote by $\RR_t(u_0,f)$ the restriction of~$\RR(u_0,f)$ at time~$t$. That is, $\RR_t$~takes $(u_0,f)$ to~$u(t)$, where~$u(t,x)$ is the solution of~\eqref{1.1}--\eqref{1.3}. 

\subsection{Continuity of the resolving operator in the relaxation norm}
\label{s1.2}
In the previous subsection, we discussed the existence and uniqueness of solution for problem~\eqref{1.1}--\eqref{1.3} and the Lipschitz continuity of the resolving operator. It turns out that the latter property remains true if the right-hand side is endowed with a weaker norm in~$t$ and a stronger norm in~$x$. Namely, define the {\it relaxation norm\/}
\begin{equation} \label{1.5}
\nnorm{f}{s}=\sup_{t\in J_T}\,\biggl\|\int_0^tf(r)\,dr\biggr\|_{H^s}
\end{equation}
on the space $L^1(J_T,H^s)$ and denote by~$B_s(R)$ the set of functions $f\in L^1(J_T,H^s)$ such that $\nnorm{f}{s}\le R$. 

\begin{proposition} \label{p1.3}
For any positive numbers~$R$ and~$T$, there is $C>0$ such that
\begin{equation} \label{1.6}
\|\RR(u_{01},f_1)-\RR(u_{02},f_2)\|_{\XX_T}
\le C\bigl(\|u_{01}-u_{02}\|+\nnorm{f_1-f_2}{1}\bigr),
\end{equation}
where $u_{01}, u_{02}\in B_{L^2}(R)$ and $f_1,f_2\in B_1(R)$ are arbitrary functions.  
\end{proposition}

\begin{proof}
We first consider the linear equation
\begin{equation} \label{1.8}
\p_t u-\nu\p_x^2u=f(t,x)
\end{equation}
supplemented with the zero initial and boundary conditions. By Theorem~\ref{t1.1}, this problem has a unique solution $Kf\in\XX_T$ for any $f\in L^1(J_T,H^1)$, which can be written in the form
\begin{equation}\label{1.9}
(Kf)(t)=\int_0^te^{\nu(t-s)\p_x^2}f(s)\,ds
=F(t)+\nu\int_0^te^{\nu(t-s)\p_x^2}\p_x^2F(s)\,ds,
\end{equation}
where we set $F(t)=\int_0^tf(s)\,ds$. The function~$\p_x^2F$ belongs to  $C(J_T,H^{-1})$, and the integral in the right-most term of~\eqref{1.9} is a solution of~\eqref{1.8} with $f=\p_x^2F$. Since the mapping $f\mapsto \p_x^2F$ is continuous from the space~$L^1(J_T,H^1)$ (endowed with the norm~$\nnorm{\cdot}{1}$) to~$L^1(J_T,H^{-1})$, recalling Remark~\ref{r1.2}, we see that the mapping $f\mapsto Kf$ is continuous from~$L^1(J_T,H^1)$ to~$\XX_T$. 

We now turn to the non-linear equation~\eqref{1.1}. Its solution can be written in the form $u=Kf+v$, where~$v\in\XX_T$ is the solution of the problem
$$
\p_t v-\nu\p_x^2v+(v+Kf)\,\p_x(v+Kf)=0, \quad v(0)=u_0.
$$
By Remark~\ref{r1.3}, this problem has a unique solution $v\in\XX_T$. Moreover, $v\in\XX_T$ is a Lipschitz function of the pair~$(u_0,Kf)$ varying in the space~$L^2\times\XX_T$. As was shown above, the mapping $f\mapsto Kf$ is continuous from the space~$L^1(J_T,H^1)$ (with the norm~$\nnorm{\cdot}{1}$) to~$\XX_T$. Hence, we obtain the required Lipschitz-continuity of the mapping $\RR(u_0,f)$. 
\end{proof}

In what follows, we shall need an analogue of Proposition~\ref{p1.3} for Eq.~\eqref{1.0} in the case when the right-hand side is endowed with the weaker norm~$\nnorm{\cdot}{0}$. In this situation, the resolving operator is only 
H\"older continuous in~$f$. The following result is one of the key points of the theory developed in the next two sections. 

\begin{proposition} \label{p1.5}
Let $u_i\in\XX_T$, $i=1,2$ be solutions of problem~\eqref{1.0}, \eqref{1.2}, \eqref{1.3} corresponding to some data $u_{0i}\in L^2$, $v_i,w_i\in L^2(J_T,H^2)$, and $f_i\in L^2(J_T,L^2)$ that belong to the balls of radius~$R$ centred at zero in the corresponding functional spaces. Then there is a constant~$C>0$ depending only on~$R$ and~$T$ such that
\begin{multline} \label{1.10}
\|u_1-u_2\|_{\XX_T}\le C\bigl(\|u_{01}-u_{02}\|+\nnorm{f_1-f_2}{0}^{1/3}\\
+\|v_1-v_2\|_{L^2(J_T,H^2)}+\|w_1-w_2\|_{L^2(J_T,H^2)}\bigr). 
\end{multline}
\end{proposition}

\begin{proof}
Let us represent a solution $u$ of Eq.~\eqref{1.0} in the form $u=Kf+\tilde u$,
where the linear operator~$K$ is defined in the proof of Proposition~\ref{p1.3} (see~\eqref{1.9}). Then~$\tilde u$ must satisfy the equation
$$
\p_tu-\nu\p_x^2(u+w)+(u+v+Kf)\,\p_x(u+v+Kf)=0
$$
and the initial--boundary conditions~\eqref{1.2}, \eqref{1.3}. Therefore, applying Remark~\ref{r1.3}, we see that 
\begin{multline*} 
\|\tilde u_1- \tilde u_2\|_{\XX_T}\le C\bigl(\|u_{01}-u_{02}\|
+\|Kf_1-Kf_2\|_{\XX_T}\\
+\|v_1-v_2\|_{L^2(J_T,H^2)}+\|w_1-w_2\|_{L^2(J_T,H^2)}\bigr). 
\end{multline*}
Thus, the required inequality~\eqref{1.10} will be established if we prove that, for any~$R$ and~$T$, there is a constant~$C_1>0$ such that 
\begin{equation} \label{2.31}
\|Kf\|_{\XX_T}\le C_1\nnorm{f}{0}^{1/3},
\end{equation}
where $f\in L^2(J_T,L^2)$ is an arbitrary function whose norm is bounded by~$R$. 

To this end, note that 
\begin{equation} \label{1.12}
\|Kf\|_{C(J_T,H^1)}+\|Kf\|_{L^2(J_T,H^2)}\le C_2. 
\end{equation}
Furthermore, we have the interpolation inequalities
$$
\|z\|\le C_3\|z\|_1^{1/2}\|z\|_{-1}^{1/2},
\quad \|z\|_1\le C_3\|z\|_2^{2/3}\|z\|_{-1}^{1/3}, 
\quad z\in H^2\cap H_0^1.
$$
Combining this with~\eqref{1.12}, we obtain
\begin{align*}
\|Kf\|_{\XX_T}&= \|Kf\|_{C(J_T,L^2)}+\|Kf\|_{L^2(J_T,H^1)}\\
&\le C_4\Bigl(\|Kf\|_{C(J_T,H^{-1})}^{1/2}+\|Kf\|_{L^2(J_T,H^{-1})}^{1/3}\Bigr).
\end{align*}
Thus, to prove~\eqref{2.31}, it suffices to show that
$$
\|Kf\|_{C(J_T,H^{-1})}\le C_5\nnorm{f}{0}.
$$
This is a consequence of~\eqref{1.9} and the inequality $\|\p_x^2e^{\tau \p_x^2}\|_{\LL(L^2,H^{-1})}\le C_6\tau^{-1/2}$, which is true for $\tau>0$. The proof is complete. 
\end{proof}

\section{Approximate controllability}
\label{s2}
\subsection{Formulation of the result and scheme of its proof}
\label{s2.1}
Let us consider Eq.~\eqref{2.1}, in which $h\in L_{\rm loc}^1(\R_+,L^2)$ is a given function and~$\eta$ is a control. We fix an arbitrary number~$T>0$ and a subspace $E\subset L^2$.

\begin{definition} \label{d2.1}
We shall say that Eq.~\eqref{2.1}  is {\it approximately controllable at time~$T$ by an $E$-valued control\/} if for any $u_0,\hat u\in L^2$ and any $\e>0$ there is $\eta\in L^2(J_T,E)$ such that 
\begin{equation} \label{2.2}
\|\RR_T(u_0,h+\eta)-\hat u\|<\e. 
\end{equation}
\end{definition}

The following theorem shows that the approximate controllability is true for any positive time with a control function taking values in a two-dimensional space. 

\begin{theorem} \label{t2.1}
Let $h\in L_{\rm loc}^1(\R_+,L^2)$ and let~$E$ be the vector span of the functions~$\sin x$ and $\sin2x$. Then Eq.~\eqref{2.1}  is approximately controllable at any time~$T$ by an $E$-valued control. 
\end{theorem}

This result is proved in Section~\ref{s2.2}--\ref{s2.5}. Here we present the scheme of the proof. 

\begin{proof}[Outline of the proof of Theorem~\ref{t2.1}]
Let us fix positive numbers~$T$ and~$\e$, arbitrary functions $u_0,\hat u\in L^2$, and a finite-dimensional space $G\subset H_0^1\cap H^2$. We shall say that Eq.~\eqref{2.1} is {\it $\e$-controllable by a $G$-valued control\/} (for given data $u_0,\hat u$, and~$T$) if there exists $\eta\in L^2(J_T,G)$ such that~\eqref{2.2} holds. Theorem~\ref{t2.1} will be established if we show that, for any $u_0,\hat u\in L^2$, Eq.~\eqref{2.1} is $\e$-controllable by an $E$-valued control. The proof of this fact is divided into four steps. 

\medskip
\underline{\sl Step~1: Extension principle}. 
Along with~\eqref{2.1}, consider the equation
\begin{equation} \label{2.3}
\p_t u-\nu \p_x^2(u+\zeta(t,x))+(u+\zeta(t,x))\p_x(u+\zeta(t,x))=h(t,x)+\eta(t,x), 
\end{equation}
where $\eta$ and $\zeta$ are $G$-valued controls. We say that Eq.~\eqref{2.3} is {\it $\e$-controllable by $G$-valued controls\/} if there are functions $\eta,\zeta\in L^2(J_T,G)$ such that the solution $u\in\XX_T$ of~\eqref{2.3}, \eqref{1.2}, \eqref{1.3} satisfies the inequality
\begin{equation} \label{2.4}
\|u(T)-\hat u\|<\e.
\end{equation}
Even though Eq.~\eqref{2.3} is ``more controlled'' than Eq.~\eqref{2.1},
it turns out that the property of $\e$-controllability
is equivalent for them. Namely, we have the following result.

\begin{proposition} \label{p2.3}
For any finite-dimensional subspace $G\subset H_0^1\cap H^2$ and any functions $u_0, \hat u\in L^2$, Eq.~\eqref{2.1} is $\e$-controllable by a $G$-valued control if and only if so is Eq.~\eqref{2.3}. 
\end{proposition}

\medskip
\underline{\sl Step~2: Convexification principle}. 
Now let $N\subset H^2\cap H_0^1$ be another finite-dimensional subspace such that
\begin{equation} \label{2.5}
 N\subset G, \quad B(N)\subset G,
\end{equation}
where $B(u)=u\p_x u$. Denote by $\FF(N,G)$ the intersection of $H^2\cap H_0^1$ with the vector space spanned by the functions of the form\footnote{Note that a function of the form~\eqref{2.6} does not necessarily belong to~$H^2\cap H_0^1$, and therefore the space $\FF(N,G)$ may coincide with~$G$.} 
\begin{equation} \label{2.6}
\eta+\xi\p_x\xi'+\xi'\p_x\xi,
\end{equation}
where $\eta,\xi\in G$ and $\xi'\in N$. It is easy to see that $\FF(N,G)\subset H^2\cap H_0^1$ is a well-defined finite-dimensional space containing~$G$.
The following proposition, which is an infinite-dimensional analogue of
the well-known convexification principle for controlled ODE's (e.g.,
see~\cite[Theorem~8.7]{AS2004}), is a key point of the proof of
Theorem~\ref{t2.1}.

\begin{proposition} \label{p2.4}
Let $N,G\subset H^2\cap H_0^1$ be finite-dimensional subspaces satisfying inclusions~\eqref{2.5}. Then~\eqref{2.3} is $\e$-controllable by $G$-valued controls if and
only if~\eqref{2.1}  is $\e$-controllable by an $\FF(N,G)$-valued control.
\end{proposition}

\medskip
\underline{\sl Step~3: Saturating property}. 
Propositions~\ref{p2.3} and~\ref{p2.4} imply the following result, which is a kind of ``relaxation property'' for  the controlled Navier--Stokes system. 

\begin{proposition} \label{p2.5}
Let $N,G\subset H^2\cap H_0^1$ be finite-dimensional subspaces satisfying inclusions~\eqref{2.5}. Then~\eqref{2.1} is $\e$-controllable by a $G$-valued control if and only if it is $\e$-controllable by an $\FF(N,G)$-valued control.
\end{proposition}

We now introduce the subspaces $E_k=\{\sin(jx), 1\le j\le k\}$, so
that the space~$E$ defined in Theorem~\ref{t2.1} coincides
with~$E_2$. We wish to apply Proposition~\ref{p2.5} to the subspaces
$N=E_1$ and $G=E_k$.

\begin{lemma} \label{l2.7}
For any integer $k\ge2$, we have $\FF(E_1,E_k)=E_{k+1}$. 
\end{lemma}

Proposition~\ref{p2.5} and Lemma~\ref{l2.7} imply that Eq.~\eqref{2.1} is $\e$-controllable by an $E_k$-valued control if and only if it is $\e$-controllable by an $E_{k+1}$-valued control. Thus, Theorem~\ref{t2.1} will be established if we find an integer $N\ge2$ such that~\eqref{2.1}  is $\e$-controllable by an $E_N$-valued control. We shall be able to do that due to the {\it saturating property\/}
\begin{equation} \label{2.7}
\bigcup_{k=2}^\infty E_k\mbox{ is dense in $L^2$},
\end{equation}
which is a straightforward consequence of the definition of~$E_k$. 

Let us mention that, in general, explicit description of the subspace~$\FF(N,G)$ and the proof of~\eqref{2.7} are difficult tasks. In our situation, it is possible to do due to the simple structures of trigonometric polynomials and of the domain on which they are studied.

\medskip
\underline{\sl Step~4: Case of a large control space}. 
It is easy to construct $\eta\in C(J_T,L^2)$ for which~\eqref{2.2} holds. Using~\eqref{2.7}, it is not difficult to
approximate~$\eta$, within any accuracy~$\delta>0$, by a function
belonging to $C(J_T, E_N)$. Since $\RR_t(u_0,\cdot)$ is continuous, what has been said implies that~\eqref{2.2} holds for an $E_N$-valued control~$\eta$. This completes the proof of Theorem~\ref{t2.1}.
\end{proof}

\subsection{Extension}
\label{s2.2}
Let us prove Proposition~\ref{p2.3}. If Eq.~\eqref{2.1}  is $\e$-controllable by a $G$-valued control, then so is~\eqref{2.3}, because one can take~$\zeta\equiv0$. Let us establish the converse assertion.

\smallskip
Let us denote by $\widehat\RR$ the resolving operator for problem~\eqref{2.3}, \eqref{1.2}, \eqref{1.3}, that is, a mapping that takes a triple $(u_0, \eta,\zeta)$ to the solution $u\in \XX_T$ of the problem in question with $h\equiv0$. By Remark~\ref{r1.3}, the operator~$\widehat\RR$ is Lipschitz continuous on bounded subsets of some appropriate functional spaces. 
Let $\hat\eta,\hat \zeta\in L^2(J_T,G)$ be arbitrary controls such that
\begin{equation} \label{3.1}
\|\widehat\RR_T(u_0,h+\hat\eta,\hat\zeta)-\hat u\|<\e,
\end{equation}
where $\widehat\RR_t$ stands for the restriction of~$\widehat\RR$ at time~$t$. 
In view of continuity of $\widehat\RR_T(u_0,h+\eta,\zeta)$ with respect to~$\zeta\in L^2(J_T,G)$, there is no loss of generality in assuming that 
\begin{equation} \label{3.2}
\hat\zeta\in C^\infty(J_T,G), \quad \hat\zeta(0)=\hat\zeta(T)=0.
\end{equation}
Consider the function $u(t)=\widehat\RR_t(u_0,h+\hat\eta,
\hat\zeta)+\hat\zeta(t)$. It is straightforward to see that it
belongs to the space~$\XX_T$ and satisfies Eqs.~\eqref{2.1}, \eqref{1.2}, \eqref{1.3} with $\eta=\hat\eta+\p_t\hat\zeta\in L^2(J_T,G)$. Moreover, it follows
from~\eqref{3.1} and~\eqref{3.2} that 
$$ 
u(0)=u_0, \quad \|u(T)-\hat u\|
=\|\widehat\RR_T(u_0,h+\hat\eta, \hat\zeta)-\hat u\|<\e.  
$$ 
Thus, Eq.~\eqref{2.1} is $\e$-controllable by a $G$-valued control. 
 
\subsection{Convexification}
\label{s2.3}
Let us prove Proposition~\ref{p2.4}. It follows from the extension
principle that if Eq.~\eqref{2.3} is $\e$-controllable
by $G$-valued controls, then~\eqref{2.1} is
$\e$-controllable by a $G$-valued control and all the more by an
$\FF(N,G)$-valued control. The proof of the converse assertion is
divided into several steps. We need to show that if $\eta_1:J_T\to\FF(N,G)$ is a square-integrable function such that
\begin{equation} \label{2.10}
\|\RR_T(u_0,h+\eta_1)-\hat u\|<\e,
\end{equation}
then there are $\eta,\zeta\in L^2(J_T,G)$ such that
\begin{equation} \label{2.11}
\|\widehat\RR_T(u_0,h+\eta,\zeta)-\hat u\|<\e.
\end{equation}

\medskip
{\sl Step~1}. 
We first show that it suffices to consider the case in which~$\eta_1$
is a piecewise constant function. Indeed, suppose
Proposition~\ref{p2.4} is proved in that case and denote
$G_1=\FF(N,G)$. For a given $\eta_1\in L^2(J_T,G_1)$, we can find a
sequence~$\{\eta^m\}$ of piecewise constant $G_1$-valued functions
such that
$$
\|\eta_1-\eta^m\|_{L^2(J_T,G_1)}\to0\quad\mbox{as $m\to\infty$}.
$$
By continuity of~$\RR_t$, there is an integer $n\ge1$ such that
\begin{equation} \label{43}
\|\RR_T(u_0,h+\eta^n)-\hat u\|<\e.
\end{equation}
Since the result is true in the case of piecewise constant controls, we can find $\eta,\zeta\in L^2(J_T,G)$ such that~\eqref{2.11} holds.

\medskip
{\sl Step~2}. We now consider the case of piecewise constant
$G_1$-valued controls. A simple iteration argument combined with the
continuity of~$\RR_t$ and~$\widehat \RR_t$ shows that it suffices to
consider the case of one interval of constancy. Thus, we shall assume
that $\eta_1(t)\equiv\eta_1\in G_1$.

We shall need the lemma below, whose proof is given at the end of this subsection. Recall that $B(u)=u\p_x u$. 

\begin{lemma} \label{l5.1}
For any $\eta_1\in\FF(N,G)$ and any $\delta>0$ there is an integer
$k\ge1$, numbers $\alpha_j>0$, and vectors $\eta,\zeta^j\in G$,
$j=1,\dots,k$, such that
\begin{align}
\sum_{j=1}^k\alpha_j&=1,\label{45}\\
\Bigl\|\eta_1-B(u)-\Bigl(\eta-\sum_{j=1}^k\alpha_j\bigl(B(u+\zeta^j)-\nu\p_x^2\zeta^j\bigr)\Bigr)\Bigr\|
&\le\delta\quad\mbox{for any $u\in H^1$}.\label{46}
\end{align}
\end{lemma}

We fix a small $\delta>0$ and choose numbers~$\alpha_j>0$ and
vectors $\eta,\zeta^j\in G$ satisfying~\eqref{45}, \eqref{46}. Let us
consider the equation
\begin{equation} \label{49}
\p_t u-\nu \p_x^2u+\sum_{j=1}^k\alpha_j\bigl(B(u+\zeta^j(x))-\nu\p_x^2\zeta^j(x)\bigr)=h(t,x)+\eta(x).
\end{equation}
This is a Burgers-type equation, and using the same arguments as in
the case of the Burgers equation, it can be proved that
problem~\eqref{49}, \eqref{1.2}, \eqref{1.3} has a unique solution $\tilde
u\in\XX_T$. On the other hand, we can rewrite~\eqref{49} in the form
\begin{equation} \label{47}
\p_t u-\nu \p_x^2u+u\p_xu=h(t,x)+\eta_1(x)-r_\delta(t,x),
\end{equation}
where $r_\delta(t,x)$ stands for the function under sign of norm on the left-hand side of~\eqref{46} in which $u=\tilde u(t,x)$.
Since~$\RR_t$ is Lipschitz continuous on bounded subsets, there is $C>0$ depending only on the $L^2$~norm of~$\eta_1$ such that
\begin{align*}
\|\RR_T(u_0,h+\eta_1)-\tilde u(T)\|
&=\|\RR_T(u_0,h+\eta_1)-\RR_T(u_0,h+\eta_1-r_\delta)\|\\[2pt]
&\le C\|r_\delta\|_{L^1(J_T,L^2)}\le CT\delta,
\end{align*}
where we used inequality~\eqref{46}. Combining this with~\eqref{2.10}, we see that if $\delta>0$ is sufficiently small, then
\begin{equation} \label{48}
\|\tilde u(T)-\hat u\|<\e.
\end{equation}

We shall show that there is a sequence $\zeta_m\in L^2(J_T,G)$ such that
\begin{equation} \label{50}
\|\widehat\RR_T(u_0,h+\eta,\zeta_m)-\tilde u(T)\|\to0
\quad\mbox{as $m\to\infty$}. 
\end{equation}
In this case, inequalities~\eqref{48} and~\eqref{50} with  $m\gg1$ will imply the required estimate~\eqref{2.11} in which $\zeta=\zeta_m$.

\medskip
{\sl Step~3}.
Following a classical idea, we define a sequence
$\zeta_m\in L^2(J_T,G)$ by the relation $\zeta_m(t)=\zeta(mt/T)$,
where $\zeta:\R\to G$ is a $1$-periodic  function such that
$$
\zeta(t)=\zeta^j\quad\mbox{for $0\le t-(\alpha_1+\cdots+\alpha_{j-1})< \alpha_j$, $j=1,\dots,k$}. 
$$
Let us rewrite~\eqref{49} in the form
$$
\p_t \tilde u-\nu \p_x^2(\tilde u+\zeta_m(t,x))+B(\tilde u+\zeta_m(t,x))
=h(t,x)+\eta(x)+f_m(t,x),
$$
where we set $f_m=f_{m1}+f_{m2}$,
\begin{align}
f_{m1}(t,x)&=-\nu \p_x^2\zeta_m+\nu\sum_{j=1}^k\alpha_j\p_x^2\zeta^j,
\label{01}\\
f_{m2}(t,x)&=B(\tilde u+\zeta_m)-\sum_{j=1}^k\alpha_j B(\tilde u+\zeta^j).
\label{02}
\end{align}
Note that the sequence $\{f_m\}$ is bounded in $L^2(J_T,L^2)$. 
Therefore, by Proposition~\ref{p1.5}, we have 
$$
\|\widehat\RR_T(u_0,h+\eta,\zeta_m)
-\widehat\RR_T(u_0,h+\eta+f_m,\zeta_m)\|\le C\,\nnorm{f_m}{0}^{1/3}. 
$$
Since $\tilde u(T)=\widehat\RR_T(u_0,h+\eta+f_m,\zeta_m)$ and $f_m=f_{m1}+f_{m2}$, convergence~\eqref{50} will be established if we prove that
\begin{equation} \label{2.22}
\nnorm{f_{m1}}{0}+\nnorm{f_{m2}}{0}\to0
\quad\mbox{as $m\to\infty$}. 
\end{equation}

\medskip
{\sl Step~4}. 
We first estimate the norm of~$f_{m1}$. The definition of~$\zeta_m$ implies that 
$$
\int_{t_{k-1}}^{t_k}f_{m1}(s)\,ds=0\quad\mbox{for any integer $k\ge1$},
$$
where $t_k=kT/m$. It follows that 
$$
\int_0^tf_{m1}(s)\,ds=\int_{\hat t_m}^tf_{m1}(s)\,ds,
$$
where $\hat t_m$ is the largest number~$t_k$ that does not exceed~$t$. Since $f_{m1}(t)$ is bounded as a function with range in~$H^2$, we conclude that
\begin{equation} \label{2.23}
\nnorm{f_{m1}}{0}= \sup_{t\in J_T}\,\biggl\|\int_{\hat t_m}^tf_{m1}(s)\,ds\biggr\|
\le C_1\sup_{t\in J_T}|t-\hat t_m|\le C_2m^{-1}. 
\end{equation}

We now turn to the estimate for~$f_{m2}$. If the function~$\tilde u$ was independent of time, we could apply an argument similar to the one used above. However, this is not the case, and to prove the required estimate, we shall approximate~$\tilde u$ by piecewise constant functions. Namely, it is easy to see that the operator~$B$ is Lipschitz continuous from $L^2(J_T,H^1)$ to $L^1(J_T,L^2)$. It follows that for any~$\e>0$ there is a piecewise constant function $\tilde u_\e:J_T\to H_0^1$ such that
$$
\|f_{m2}-f_{m2}^\e\|_{L^1(J_T,L^2)}\le\e,
$$
where $f_{m2}^\e$ stands for the function given by~\eqref{02} with $\tilde u=\tilde u_\e$. It follows that $\nnorm{f_{m2}-f_{m2}^\e}{0}\le T\e$, and hence we can assume from the very beginning that~$\tilde u$ is piecewise constant. In other words, there is a partition $0=\tau_0<\tau_1<\cdots<\tau_N=T$ of the interval~$[0,T]$ and functions $u_n\in H_0^1$, $n=1,\dots,N$,  such that 
$$
f_{m2}(t,x)=B(u_n+\zeta_m)-\sum_{j=1}^k\alpha_j B(u_n+\zeta^j)
\quad\mbox{for $\tau_{n-1}\le t<\tau_n$}.
$$
Now note that if $[t_{k-1},t_k]\subset[\tau_{n-1},\tau_n]$, then
$$
\int_{t_{k-1}}^{t_k}f_{m2}(t,x)\,dt=0.
$$
Repeating the argument used for~$f_{m1}$, we easily prove that 
$\nnorm{f_{m2}}{0}\le C_3m^{-1}$ in the case when~$\tilde u$ is piecewise constant. Combining this with~\eqref{2.23}, we obtain the required convergence~\eqref{2.22}. The proof of Proposition~\ref{p2.4} is complete. 

\begin{proof}[Proof of Lemma~\ref{l5.1}]
It suffices to find functions $\eta,\tilde\zeta^j\in G$, $j=1,\dots,m$, such that
\begin{equation} \label{71}
\Bigl\|\eta_1-\eta+\sum_{j=1}^mB(\tilde\zeta^j)\Bigr\|\le\delta.
\end{equation}
If such vectors are constructed, then we can set $k=2m$,
$$
\alpha_j=\alpha_{j+m}=\frac{1}{2m},\quad 
\zeta^j=-\zeta^{j+m}=\sqrt{m}\,\tilde\zeta^j\quad\mbox{for $j=1,\dots,m$}. 
$$

To construct $\eta,\tilde\zeta^j\in G$ satisfying~\eqref{71}, note that if $\eta_1\in\FF(N,G)$, then there are functions $\tilde\eta_j,\xi_j\in G$ and $\xi'_j\in N$ such that
\begin{equation} \label{72}
\eta_1=\sum_{j=1}^m\bigl(\tilde\eta_j-\xi_j\p_x\xi'_j-\xi'_j\p_x\xi_j\bigr).
\end{equation}
Now note that, for any $\e>0$, 
$$
\xi_j\p_x\xi'_j+\xi'_j\p_x\xi_j
=B(\e\xi_j+\e^{-1}\xi'_j)-\e^2B(\xi_j)-\e^{-2}B(\xi'_j).
$$
Combining this with~\eqref{72}, we obtain
$$
\eta_1-\sum_{j=1}^m\bigl(\tilde\eta_j+\e^{-2}B(\xi'_j)\bigr)
+\sum_{j=1}^mB(\e\xi_j+\e^{-1}\xi'_j)=\e^2\sum_{j=1}^mB(\xi_j).
$$
Choosing $\e>0$ sufficiently small and setting
$$
\eta=\sum_{j=1}^m\bigl(\tilde\eta_j+\e^{-2}B(\xi'_j)\bigr), \quad
\tilde\zeta^j=\e\xi_j+\e^{-1}\xi'_j, 
$$
we arrive at~\eqref{71}. 
\end{proof}

\subsection{Saturation}
\label{s2.4}
Let us prove Lemma~\ref{l2.7} and the inclusion $B(E_1)\subset E_2$. For $\xi=\sin(jx)$ and $\xi'=\sin x$, we have
\begin{align} 
\xi\p_x\xi'+\xi'\p_x\xi
&=\sin(jx)\cos x+j\sin x\cos(jx)\notag\\
&=\frac12\bigl((j+1)\sin(j+1)x-(j-1)\sin(j-1)x\bigr).\label{3.3}
\end{align}
It follows that $B(E_1)\subset E_2$ and $\FF(E_1,E_k)\subset E_{k+1}$. Furthermore, taking $j=k$ in~\eqref{3.3}, we write
$$
\sin(k+1)x=\frac{k-1}{k+1}\sin(k-1)x+\frac{2}{k+1}\bigl(\sin(kx)\,\p_x\sin x+\sin x\,\p_x\sin(kx)\bigr). 
$$
This relation implies that the function $\sin(k+1)x$ belongs to~$\FF(E_1,E_k)$ and therefore $E_{k+1}\subset\FF(E_1,E_k)$.

\subsection{Case of a large control space}
\label{s2.5}
We wish to construct a control $\eta\in L^2(J_T,E_N)$ with a large integer~$N\ge2$ such that~\eqref{2.2} holds. To this end, consider a function $u_\mu$ defined as
$$
u_\mu(t,x)=T^{-1}\bigl(te^{\mu \p_x^2}\hat u+(T-t)e^{t \p_x^2}u_0\bigr),
$$
where $\mu>0$ is a small number that will be chosen below. The function~$u_\mu$ belongs to the space~$\XX_T$ and satisfies Eqs.~\eqref{2.1}, \eqref{1.2}, \eqref{1.3} in which
$$
\eta=\eta_\mu:=\p_tu_\mu-\nu\p^2_x u_\mu+u_\mu\p_xu_\mu-h.
$$
This function belongs to~$L^1(J_T,L^2)$. Furthermore, 
\begin{equation} \label{3.4}
\|u_\mu(T)-\hat u\|=\|e^{\mu \p_x^2}\hat u-\hat u\|\to0\quad \mbox{as $\mu\to0$}. 
\end{equation}
Choosing $\mu>0$ sufficiently small in~\eqref{3.4} and
approaching~$\eta_\mu\in L^1(J_T,L^2)$ by continuous $L^2$-valued
functions, we can find $\tilde\eta\in C(J_T,L^2)$ such that 
\begin{equation} \label{3.5}
\|\RR_T(u_0,h+\tilde\eta)-\hat u\|<\e.
\end{equation}

Let us denote by ${\mathsf P}_k:L^2\to L^2$ the orthogonal projection
in~$L^2$ onto the subspace~$E_k$. In view of the saturating
property~\eqref{2.7}, we have 
$$
\sup_{t\in[0,T]}\|{\mathsf P}_k\tilde\eta(t)-\tilde\eta(t)\|\to0\quad\mbox{as $k\to\infty$.}
$$
By continuity of~$\RR_t$, we obtain
$$
\|\RR_T(u_0,h+{\mathsf P}_k\tilde\eta)-\RR_T(u_0,h+\tilde\eta)\|\to0\quad\mbox{as $k\to\infty$.}
$$
Combining this with~\eqref{3.5}, we see that, for a sufficiently large $N\ge1$, the function
$\eta={\mathsf P}_N\tilde\eta$ satisfies~\eqref{2.2}. This completes
the proof of Theorem~\ref{t2.1}.

\section{Exact controllability of finite-dimensional functionals}
\label{s3}
\subsection{Main result}
\label{s3.1}
Let us introduce a controllability property which is stronger than the approximate  controllability. To this end, we first define the concept of a regular point for a continuous function. 

\begin{definition} \label{d3.1}
Let $X$ be a Banach space and let $F\!:X\to\R^N$ be a continuous function. We shall say that $\hat u\in X$ is a {\it regular point for~$F$} if there is a non-degenerate closed ball $B\subset\R^N$ centred at~$\hat y=F(\hat u)$ and a continuous mapping\footnote{Let us emphasise that $F^{-1}$ is just a notation.} $F^{-1}:B\to X$ such that $F^{-1}(\hat y)=\hat u$ and~$F^{-1}$ is the right inverse of~$F$ on~$F^{-1}(B)$:
\begin{equation} \label{5.1}
F(F^{-1}(y))=y\quad\mbox{for $y\in B$}. 
\end{equation}
\end{definition}

For instance, if $F:X\to\R^N$ is an analytic function such that $F(X_0)$ contain an open ball for some finite-dimensional affine subspace $X_0\subset X$, then the Sard theorem implies that almost every point $\hat u\in X_0$ is regular for~$F$. In particular, if~$F$ is a finite-dimensional projection in~$X$, then any point is regular for~$F$.

\begin{definition} \label{d3.2}
Let $E\subset L^2$ be a closed subspace. We shall say that the Burgers equation~\eqref{2.1} is {\it controllable at time~$T>0$ by an $E$-valued control\/} if for any continuous function $F:L^2\to\R^N$ the following property holds: for any initial function $u_0\in L^2$, any regular point $\hat u\in L^2$, and any $\e>0$ there is $\eta\in C^\infty(J_T,E)$ such that
\begin{gather}
\|\RR_T(u_0,h+\eta)-\hat u\|<\e, \label{5.2}\\
F\bigl(\RR_T(u_0,h+\eta)\bigr)=F(\hat u). \label{5.3}
\end{gather}
\end{definition}

Thus, the controllability property is stronger than the exact controllability in observed projection (cf.~\cite{AS-2005,AS-2008}), but is much weaker than the usual concept of exact controllability. 

\begin{theorem} \label{t3.3}
Let $h$ and $E$ be the same as in Theorem~\ref{t2.1}. Then Eq.~\eqref{2.1} is controllable at any time~$T>0$ by an $E$-valued control. 
\end{theorem}

The proof of this result is outlined in the next subsection, and the details are given in Sections~\ref{s3.5}--\ref{s3.4}. 

\subsection{Reduction to a uniform approximate controllability}
\label{s3.2}
The proof of Theorem~\ref{t3.3} is based on the property of {\it uniform approximate controllability\/}. 

\begin{definition} \label{d3.4}
We shall say that Eq.~\eqref{2.1} is {\it uniformly approximately controllable at time~$T$ by an $E$-valued control\/} if for any $\e>0$ and any compact set $\KK\subset L^2$ there is a continuous mapping $\varPsi:\KK\times \KK\to L^2(J_T,E)$ such that 
\begin{gather}
\varPsi(\KK\times\KK)\subset C^\infty(J_T,E),\label{5.4}\\
\sup_{u_0,\hat u\in \KK}\,\bigl\|\RR_T(u_0,h+\varPsi(u_0,\hat u))-\hat u\bigr\|<\e. 
\label{5.5}
\end{gather}
\end{definition}

Thus, the uniform approximate controllability can be regarded as a parameter version of the approximate controllability. The following result is an analogue of Theorem~\ref{t2.1} for this concept. 

\begin{theorem} \label{t3.5}
Under the hypotheses of Theorem~\ref{t2.1}, Eq.~\eqref{2.1} is uniformly approximately controllable  at any time $T>0$ by an $E$-valued control. 
\end{theorem}

We claim that if Eq.~\eqref{2.1} is uniformly approximately controllable at time~$T$ by an $E$-valued control, then it is controllable. Indeed, let $\hat u\in L^2$ be a regular point for a continuous function $F:L^2\to\R^N$, let $u_0\in L^2$ be an initial function, and let~$\e>0$. We wish to construct a control $\eta\in C^\infty(J_T,E)$ such that~\eqref{5.2} and~\eqref{5.3} hold. 

By the definition of a regular point, there is a ball $B\subset\R^N$ centred at the point $\hat y=F(\hat u)$ and a continuous function $F^{-1}:B\to L^2$ such that $F^{-1}(\hat y)=\hat u$ and~\eqref{5.1}  holds. Without loss of generality, we can assume that the radius~$r$ of the ball~$B$ is so small that 
\begin{equation} \label{5.6}
\sup_{y\in B}\|F^{-1}(y)-\hat u\|<\frac{\e}{2}.
\end{equation}
Denote $\KK=F^{-1}(B)\cup\{u_0\}$, so that~$\KK$ is a compact subset of~$L^2$. Let us choose a number $\delta\in(0,\e/2)$ such that 
\begin{equation} \label{5.7}
\|F(u_1)-F(u_2)\|\le r\quad
\mbox{for $u_1,u_2\in \KK$, $\|u_1-u_2\|\le\delta$}. 
\end{equation}
Theorem~\ref{t3.5} implies that there is a continuous mapping $\varPsi:\KK\to L^2(J_T,E)$ with range in $C^\infty(J_T,E)$ such that
\begin{equation} \label{5.8}
\sup_{v\in\KK}\bigl\|\RR_T\bigl(u_0,h+\varPsi(v)\bigr)-v\bigr\|<\delta. 
\end{equation}
Consider the mapping $\varPhi:B\to\R^N$ defined by 
$$
\varPhi(y)=F\bigl(\RR_T(u_0,h+\varPsi\circ F^{-1}(y))\bigr).
$$
It follows from~\eqref{5.7} that
$$
\sup_{y\in B}\|\varPhi(y)-y\|=\sup_{y\in B}
\bigl\|F\bigl(\RR_T(u_0,h+\varPsi\circ F^{-1}(y))\bigr)-F\bigl(F^{-1}(y)\bigr)\bigr\|
\le r. 
$$
Thus, applying the Brouwer theorem to the mapping $\Gamma:B\to B$ taking~$y$ to~$y-\varPhi(y)+\hat y$, we can find $\bar y\in B$ such that $\varPhi(\bar y)=\hat y$. This equality coincides with relation~\eqref{5.3} in which $\eta=\varPsi\circ F^{-1}(\bar y)$. Furthermore, setting $\bar u=F^{-1}(\bar y)$ and using~\eqref{5.6} and~\eqref{5.8}, we obtain
$$
\|\RR_T(u_0,h+\eta)-\hat u\|
\le \|\RR_T(u_0,h+\varPsi(\bar u))-\bar u\|+\|F^{-1}(\bar y)-\hat u\|
<\delta+\frac{\e}{2}<\e. 
$$

Thus, it suffices to prove Theorem~\ref{t3.5}.  To this end, we repeat the scheme used in Section~\ref{s2}, following carefully the dependence of controls on the initial and final points. Namely, let us fix~$\e>0$, a compact set $\KK\subset L^2$, and a finite-dimensional subspace $G\subset L^2$. We say that Eq.~\eqref{2.1} is {\it $(\e,\KK)$-controllable by a $G$-valued control\/} if there is a continuous mapping $\varPsi:\KK\times \KK\to L^2(J_T,G)$ satisfying~\eqref{5.4} with $E=G$ and~\eqref{5.5}. We shall prove that some analogues of Propositions~\ref{p2.3} and~\ref{p2.4} are true for $(\e,\KK)$-controllability. Once they are established, the required result will follow from the saturating property and the fact that~\eqref{2.1} is $(\e,\KK)$-controllable by an $E_N$-valued control with a sufficiently large~$N$. 

The realisation of the above scheme is based on a result on uniform approximation of solutions for a Burgers-type equation. It is given in the next subsection. The proof of Theorem~\ref{t3.5} is presented in Sections~\ref{s3.3} and~\ref{s3.4}. 

\subsection{Uniform approximation of solutions}
\label{s3.5}
Let $(\CC,d_\CC)$ be a compact metric space and let $b_i:\CC\to\R_+$, $i=1,\dots,q$, be continuous functions such that
\begin{equation} \label{5.41}
\sum_{i=1}^q b_i(y)=1\quad\mbox{for all $y\in\CC$}. 
\end{equation}
Let us fix some functions $\zeta^i\in H^2\cap H_0^1$, $i=1,\dots,q$, and  consider the following Burgers-type equation depending on the parameter $y\in\CC$:
\begin{equation} \label{5.42}
\p_t u-\nu \p_x^2u+\sum_{i=1}^q
b_i(y)\bigl(B(u+\zeta^i(x))-\nu\p_x^2\zeta^i(x)\bigr)=f(t,x).
\end{equation}
For any $y\in \CC$ and $u_0\in L^2$, this equation has a unique solution $u\in \XX_T$ issued from~$u_0$. Let us denote by $\sS:\CC\times L^2\times L^1(J_T,L^2)\to\XX_T$ a mapping that takes the triple $(y,u_0,f)$ to the solution~$u$ of problem~\eqref{5.42}, \eqref{1.2}. Recall that~$\widehat\RR$ stands for the resolving operator of Eq.~\eqref{2.3}. The following result shows that the solutions of~\eqref{5.42} can be approximated by those of~\eqref{2.3}. 

\begin{proposition} \label{p5.6}
Under the above hypotheses, for any positive numbers~$R$, $T$, and~$\e$ there is a continuous function $\varPsi:\CC\to L^2(J_T, H^2)$ such that 
\begin{gather}
\varPsi(t;y)\in\{\zeta^1,\dots,\zeta^q\}\quad
\mbox{for all $y\in\CC$, $t\in J_T$}, \label{5.43}\\
\sup_{y,u_0,f}\,\bigl\|\widehat\RR\bigl(u_0,f,\varPsi(y)\bigr)
-\sS(y,u_0,f)\bigr\|_{\XX_T}\le\e,\label{5.44}
\end{gather}
where the supremum is taken over $y\in\CC$, $u_0\in L^2$, and $f\in L^1(J_T,L^2)$ such that $\|u_0\|\le R$ and $\|f\|_{L^1(J_T,L^2)}\le R$. 
\end{proposition}

\begin{proof}
We repeat the argument used in Step~3 of the proof of Proposition~\ref{p2.4}. The main point is to follow carefully the dependence on the parameter~$y$ and the functions~$u_0$ and~$f$. 

\smallskip
{\it Step~1}. 
Define a sequence of mappings $\varPsi^m:\CC\to L^2(J,H^2)$ by the formula 
$$
\varPsi^m(t;y)=\zeta(mt/T;y),
$$
where $\zeta=\zeta(t;y)$ is a $1$-periodic function depending on the parameter~$y$ such that 
$$
\zeta(t;y)=\zeta^i\quad\mbox{for}\quad 
0\le t-(b_1(y)+\cdots+b_{i-1}(y))<b_i(y), \quad i=1,\dots,q. 
$$
The continuity of the functions~$b_i$ implies that $\varPsi^m$ is also continuous. Let us denote by~$u(y)=u(y,u_0,f)\in\XX_T$ the solution of~\eqref{5.42}, \eqref{1.2} and rewrite Eq.~\eqref{5.42} in the form 
$$
\p_t u(y)-\nu \p_x^2\bigl(u(y)+\varPsi^m(y)\bigr)+B\bigl(u(y)+\varPsi^m(y)\bigr)
=f(t,x)+f_m(t,x;y,u_0,f),
$$
where $f_m(t,x;y,u_0,f)=f_{m1}(t,x;y)+f_{m2}(t,x;y,u_0,f)$, and the functions~$f_{m1}$ and~$f_{m2}$ are defined by formulas~\eqref{01} and~\eqref{02} in which~$\zeta_m$ and~$\tilde u$ are replaced by~$\varPsi^m(y)$ and $u(y,u_0,f)$, respectively. Since the norm of~$\varPsi^m(y)$ in $L^2(J_T,H^2)$ is bounded for $m\ge1$ and $y\in\CC$, Proposition~\ref{p1.5} implies that 
$$
\|u^m(y,u_0,f)-u(y,u_0,f)\|_{\XX_T}
\le C\,\nnorm{f_m(y,u_0,f)}{0}^{1/3},
$$
where $u^m=u^m(y,u_0,f)=\widehat\RR(u_0,f,\varPsi^m(y))$. Thus, Proposition~\ref{p5.6} will be proved if we show that 
$$
\sup_{y,u_0,f}\nnorm{f_m(y,u_0,f)}{0}\to0\quad\mbox{as}\quad m\to\infty.
$$
The fact that the relaxation norm of each function~$f_m(y,u_0,f)$ goes to zero as $m\to\infty$ was established in Step~4 of the proof of Proposition~\ref{p2.4}. To prove that the convergence is uniform in $(y,u_0,f)$, it suffices to prove that the family of mappings $\fff_m:\CC\times L^2\times L^1(J_T,L^2)\mapsto L^1(J,L^2)$ taking~$(y,u_0,f)$ to~$f_m(y,u_0,f)$ is uniformly equicontinuous, that is,
\begin{equation} \label{5.23}
\sup_{m\ge1}\|f_m(y_1,u_{01},f_1)-f_m(y_2,u_{02},f_2)\|_{L^1(J,L^2)}\to0,
\end{equation}
as $d_\CC(y_1,y_2)+\|u_{01}-u_{02}\|+\|f_1-f_2\|_{L^1(J_T,L^2)}\to0$. 

\smallskip
{\it Step~2}. 
Since the bilinear term $B(u)=u\p_xu$ is continuous from~$H^1$ to~$L^2$, it follows from relation~\eqref{02} with $\tilde u=u(y,u_0,f)$ and $\zeta_m=\varPsi^m(y)$ that convergence~\eqref{5.23} will be proved if we show that 
\begin{equation} \label{5.24}
\|u(y_1,u_{01},f_1)-u(y_2,u_{02},f_2)\|_{L^2(J,H^1)}
+\sup_{m\ge1}\|\varPsi^m(y_1)-\varPsi^m(y_2)\|_{L^2(J,H^1)}\to0.
\end{equation}
The fact that the first term goes to zero follows immediately from the continuous dependence of solutions for~\eqref{5.42} on the problem data. Thus, we shall concentrate on the second term. 

In view of the definition of~$\varPsi^m$ and the periodicity of~$\zeta(t;y)$, we have
\begin{align*}
\|\varPsi^m(y_1)-\varPsi^m(y_2)\|_{L^2(J,H^1)}^2
&=\int_0^T\|\zeta(mt/T;y_1)-\zeta(mt/T;y_2)\|_1^2dt\\
&=T\int_0^1\|\zeta(t;y_1)-\zeta(t;y_2)\|_1^2dt\\
&\le C\sum_{i=1}^q|b_i(y_1)-b_i(y_2)|. 
\end{align*}
Since the continuous functions~$b_i$ are uniformly continuous on the compact space~$\CC$, we see that the second term in~\eqref{5.24} goes to zero as $d_\CC(y_1,y_2)\to0$. This completes the proof of Proposition~\ref{p5.6}. 
\end{proof}

\subsection{Extension and convexification with parameters}
\label{s3.3}
Let us consider the controlled equation~\eqref{2.3}. Given a number $\e>0$, a compact set~$\KK\subset L^2$, and a finite-dimensional subspace $G\subset H^2$, we say that Eq.~\eqref{2.3} is {\it $(\e,\KK)$-controllable by $G$-valued controls\/} if there exist two continuous functions $\varPsi_1,\varPsi_2:\KK\times\KK\to L^2(J_T,G)$ such that
\begin{gather}
\varPsi_i(\KK\times\KK)\subset C^\infty(J_T,G),\quad
i=1,2,\label{5.9}\\
\sup_{u_0,\hat u\in \KK}\,
\bigl\|\widehat\RR_T(u_0,h+\varPsi_1(u_0,\hat u),\varPsi_2(u_0,\hat u))-\hat u\bigr\|<\e.
\label{5.10}
\end{gather}
The following result is a parameter version of Proposition~\ref{p2.3}.

\begin{proposition} \label{p3.6}
Let $G\subset H_0^1\cap H^2$. Then~\eqref{2.1} is $(\e,\KK)$-controllable by a $G$-valued control if and only if so is~\eqref{2.3}.
\end{proposition}

\begin{proof}
Let $\varPsi_i:\KK\times\KK\to L^2(J_T,G)$, $i=1,2$ be two mappings satisfying~\eqref{5.9} and~\eqref{5.10}. Since $C_0^\infty(J_T,G)$ is dense in $L^2(J_T,G)$, we can assume that the images of both mappings are contained in a finite-dimensional subspace of $C_0^\infty(J_T,G)$; see Proposition~\ref{p4.2}. It follows that (cf.\ proof of Proposition~\ref{p2.3})
\begin{equation} \label{5.11}
\widehat\RR\bigl(u_0,h+\varPsi_1(y),\varPsi_2(y)\bigr)
+\varPsi_2(y)
=\RR\bigl(u_0,h+\varPsi_1(y)+\p_t\varPsi_2(y)\bigr),
\end{equation}
where we set $y=(u_0,\hat u)$. Since all the norms on a finite-dimensional space are equivalent,  the mapping
$$
\varPsi:\KK\times\KK\to L^2(J_T,G), \quad y\mapsto \varPsi_1(y)+\p_t\varPsi_2(y),
$$
is continuous, and its image is contained in $C_0^\infty(J_T,G)$. Finally, combining~\eqref{5.10} and~\eqref{5.11}, we conclude that~\eqref{5.5} also holds. The proof is complete. 
\end{proof}

We now turn to a parameter version of the convexification principle. 

\begin{proposition} \label{p3.7}
Under the hypotheses of Proposition~\ref{p2.4}, Eq.~\eqref{2.3} is $(\e,\KK)$-con\-trollable by $G$-valued controls if and
only if Eq.~\eqref{2.1}  is $(\e,\KK)$-controllable by an $\FF(N,G)$-valued control.
\end{proposition}

\begin{proof}
We repeat essentially the scheme used to prove Proposition~\ref{p2.5}. The main point is to follow the dependence of all the objects on the initial and target functions~$u_0$ and~$\hat u$. 

\medskip
{\it Step 1}. To simplify notation, set $G_1=\FF(N,G)$, $\CC=\KK\times\KK$, and $y=(u_0,\hat u)$. Let us assume that $\varPsi:\CC \to L^2(J_T,G_1)$ is a continuous mapping satisfying~\eqref{5.4} with $E=G_1$ and~\eqref{5.5}. By Proposition~\ref{p4.1} and continuity of the resolving operator~$\RR$, we can construct a continuous  function $\widehat\varPsi:\CC \to L^2(J_T,G_1)$ that satisfies~\eqref{5.5} and has the form
\begin{equation} \label{5.13}
\widehat\varPsi(y)
=\sum_{r=1}^{s}\sum_{l=1}^{L}c_{lr}(y)I_{r,s}(t)\eta^l,
\end{equation}
where $L=2\dim G_1$, $\eta^1,\dots,\eta^L\in G_1$ are some vectors, and $c_{lr}:\CC \to\R$ are non-negative continuous functions such that 
$$
\sum_{l=1}^Lc_{lr}(y)\equiv 1\quad\mbox{for $r=1,\dots,s$}. 
$$
We shall prove that, given any~$\sigma>0$, one can find continuous mappings $\varPsi_i^\sigma:\CC \to L^2(J_T,G)$, $i=1,2$ such that 
\begin{equation} \label{5.12}
\sup_{y\in\CC}\bigl\|\RR_T\bigl(u_0,h+\widehat\varPsi(y)\bigr)
-\widehat\RR_T\bigl(u_0,h+\varPsi_1^\sigma(y),\varPsi_2^\sigma(y)\bigr)\bigr\|\le\sigma.
\end{equation}
Once this property is proved, for a sufficiently small $\sigma>0$ we shall have 
$$
\sup_{y\in\CC}\,\bigl\|
\widehat\RR_T\bigl(u_0,h+\varPsi_1^\sigma(y),\varPsi_2^\sigma(y)\bigr)
-\hat u\bigr\| <\e.
$$
Finally, using Proposition~\ref{p4.2}, we can find continuous functions $\varPsi_1,\varPsi_2$ from~$\CC $ to a finite-dimensional subspace of $C_0^\infty(J_T,G)$ such that~\eqref{5.10} holds. Thus, it suffices to prove~\eqref{5.12}. 

\smallskip
{\it Step 2}.
We first assume that $s=1$, that is, there is only one interval of constancy. In this case, we can rewrite~\eqref{5.13} as
\begin{equation} \label{5.27}
\widehat\varPsi(y)
=\sum_{l=1}^{L}c_{l}(y)\eta^l.
\end{equation}
Applying Lemma~\ref{l5.1} to the functions~$\eta^l$, for any $\delta>0$ we can find numbers $\alpha_{jl}\ge0$ and vectors $\xi^l,\zeta^{jl}\in G$ such that (cf.~\eqref{45}, \eqref{46})
\begin{align}
\sum_{j=1}^k\alpha_{jl}&=1,\label{5.14}\\
\Bigl\|\eta^l-B(u)-\Bigl(\xi^l-\sum_{j=1}^k\alpha_{jl}\bigl(B(u+\zeta^{jl})-\nu\p_x^2\zeta^{jl}\bigr)\Bigr)\Bigr\|
&\le\delta\quad\mbox{for any $u\in H^1$},\label{5.15}
\end{align}
where $l=1,\dots,L$. 
Consider the equation
\begin{equation} \label{5.16}
\p_t u-\nu \p_x^2u+\sum_{j=1}^k\sum_{l=1}^L
\alpha_{jl}c_l(y)\bigl(B(u+\zeta^{jl})-\nu\p_x^2\zeta^{jl}\bigr)=h+\xi,
\end{equation}
where we set 
\begin{equation} \label{5.25}
\xi=\xi(x;y)=\sum_{l=1}^L c_l(y)\xi^l(x).
\end{equation}
Indexing the pairs $(j,l)$ by a single sequence $i=1,\dots,q$, we rewrite~\eqref{5.16} as
\begin{equation} \label{5.17}
\p_t u-\nu \p_x^2u+\sum_{i=1}^q
b_i(y)\bigl(B(u+\zeta^i(x))-\nu\p_x^2\zeta^i(x)\bigr)=h(t,x)+\xi(x;y),
\end{equation}
where $b_i$ are non-negative continuous functions whose sum is equal to~$1$. Equation~\eqref{5.17} has a unique solution $\tilde u=\tilde u(t;y)$ in~$\XX_T$  issued from~$u_0\in\KK$. On the other hand, we can rewrite~\eqref{5.17} in the form (cf.~\eqref{47})
\begin{equation} \label{5.18}
\p_t u-\nu \p_x^2u+u\p_xu=h(t,x)+\widehat\varPsi(y)-r_\delta(t,x;y),
\end{equation}
where $r_\delta$ is defined by
$$
r_\delta(y)=\widehat\varPsi(y)-B(\tilde u)
-\biggl(\xi(y)-\sum_{i=1}^q
b_i(y)\bigl(B(\tilde u+\zeta^i)-\nu\p_x^2\zeta^i\bigr)\biggr).
$$
Note that, in view of~\eqref{5.15}, we have 
$$
\sup_{y\in\CC}\|r_\delta(t;y)\|\le L\delta.
$$
Combining this with the Lipschitz continuity of~$\RR_T$ on bounded subsets, we see that
\begin{align*}
&\sup_{y\in\CC}\,
\bigl\|\RR_T(u_0,h+\widehat\varPsi(y))-\tilde u(T;y)\bigr\|\\
&\qquad=\sup_{y\in\CC}\,
\bigl\|\RR_T(u_0,h+\widehat\varPsi(y))-\RR_T(u_0,h+\widehat\varPsi(y)
-r_\delta(y))\bigr\|\\
&\qquad
\le C\sup_{y\in\CC}\,\bigl\|r_\delta(y)\bigr\|_{L^1(J_T,L^2)}
\le CTL\,\delta. 
\end{align*}
Recalling now inequality~\eqref{5.5} with~$\varPsi$ replaced by~$\widehat\varPsi$, we conclude that if~$\delta>0$ is sufficiently small, then 
$$
\sup_{y\in\CC}\,
\bigl\|\tilde u(T;y)-\hat u\bigr\|<\e. 
$$
Thus, to prove~\eqref{5.12} for $s=1$, it suffices to construct, for any given~$\sigma>0$, a  continuous mapping $\varPsi_2^\sigma:\CC \to L^2(J_T,G)$ such that 
\begin{equation} \label{5.19}
\sup_{y\in\CC }\,
\bigl\|\RR_T\bigl(u_0,h+\xi(y),\varPsi_2^\sigma(y)\bigr)-\tilde u(T;y)\bigr\|
\le\sigma.
\end{equation}
The existence of such a mapping is a straightforward consequence of Proposition~\ref{p5.6}. 

\smallskip
{\it Step 3}.
We now turn to the case $s\ge2$. Let us note that the construction of the previous step implies the following result on approximation of solutions. 

\begin{lemma} \label{l3.8}
Let $J\subset\R$ be a finite interval and let~$(\CC,d_\CC)$ be a compact metric space. Then for any elements $\eta^l\in\FF(N,G)$, $l=1,\dots,L$, any non-negative continuous functions $c_l:\CC\to\R$ whose sum is identically equal to~$1$, and any positive numbers~$\sigma$ and~$R$ there are continuous functions
$$
\varPsi_1:\CC\to G, \quad \varPsi_2:\CC\to L^2(J,G)
$$
and a number~$\delta>0$ such that, for any $u_0,v_0\in B_{L^2}(R)$ and $y\in\CC$ satisfying the inequality $\|u_0-v_0\|\le\delta$, we have
$$
\bigl\|\RR\bigl(u_0,h+\widehat\varPsi(y)\bigr)
-\widehat\RR\bigl(v_0,h+\varPsi_1(y),\varPsi_2(y)\bigr)\bigr\|_{\XX(J)}\le\sigma,
$$
where $\widehat\varPsi(y)$ is defined by~\eqref{5.27}, and with a slight abuse of notation we denote by~$\RR$ and~$\widehat\RR$ the resolving operators for~\eqref{2.1} and~\eqref{2.3} on the interval~$J$.  
\end{lemma}

Let us set $J_r=[t_{r-1},t_{r}]$, $r=1,\dots,s$, and define the restrictions of the required mappings~$\varPsi_1^\sigma$ and~$\varPsi_2^\sigma$ to~$J_r$ consecutively from~$r=s$ to~$r=1$. Namely, let positive numbers~$\e_s$ and~$R$ be such that 
$$
\e_s+\sup_{y\in\CC}
\bigl\|\RR_T\bigl(u_0,h+\widehat\varPsi(y)\bigr)-\hat u\bigr\|<\e,\quad
\sup_{y\in\CC}\,\bigl\|\RR\bigl(u_0,h+\varPsi(y)\bigr)\bigr\|\le R-1. 
$$
If $\e_{r}>0$ is constructed for some integer $r\in[2,s]$, we apply Lemma~\ref{l3.8} with $J=J_r$, $\sigma=\e_{r}$, and the above choice of~$R$ to find mappings 
$$
\varPsi_1^\sigma(r,\cdot):\CC\to G,\quad 
\varPsi_2^\sigma(r,\cdot):\CC\to L^2(J_r,G)
$$
and a number $\delta\in(0,1)$ such that, for any $v_0\in L^2$ satisfying the inequality $\|v_0-\RR_{t_{r-1}}(u_0,h+\widehat\varPsi(y))\|\le\delta$, we have 
$$
\sup_{y\in\CC}\,\bigl\|\RR\bigl(u_0,h+\widehat\varPsi(y)\bigr)
-\widehat\RR\bigl(v_0,h+\varPsi_1(r;y),\varPsi_2(r;y)\bigr)\bigr\|_{\XX(J_r)}
\le\e_r. 
$$
Setting $\e_{r-1}=\delta$, we can continue the construction up to $r=1$. 
We now define the required mappings by the relation
$$
\varPsi_1^\sigma(y)\big|_{J_r}=\varPsi_1^\sigma(r;y), \quad
\varPsi_2^\sigma(y)\big|_{J_r}=\varPsi_2^\sigma(r;y), \quad y\in\CC,\quad
r=1,\dots,s.
$$
It is easy to see that the constructed mappings satisfy the required inequality~\eqref{5.12}. 
\end{proof}

\subsection{Completion of the proof of Theorem~\ref{t3.5}}
\label{s3.4}
Propositions~\ref{p3.6} and~\ref{p3.7} combined with Lemma~\ref{l2.7} imply that Eq.~\eqref{2.1} is $(\e,\KK)$-controllable by an $E$-valued control if and only if it is $(\e,\KK)$-controllable by an $E_N$-valued control, where the spaces~$E_k$ are defined after Proposition~\ref{p2.5}. Thus, the proof of Theorem~\ref{t3.5} will be complete if we establish the latter property with a large $N\ge2$. 

\smallskip
Let $u_\mu=u_\mu(u_0,\hat u)$ and~$\eta_\mu=\eta_\mu(u_0,\hat u)$ be the functions defined in Section~\ref{s2.5}. Then~$\eta_\mu$ maps continuously~$\KK\times\KK$ to $L^2(J_T,L^2)$ and has the property that 
$$
\sup_{u_0,\hat u\in\KK}\|u_\mu(T)-\hat u\|
=\sup_{u_0,\hat u\in\KK}\|\RR_T(u_0,h+\eta_\mu(u_0,\hat u))-\hat u\|
\to0\quad\mbox{as $\mu\to0$}. 
$$
Using the density of $C^\infty(J_T,L^2)$ in the space $L^2(J_T,L^2)$ and applying Proposition~\ref{p4.2}, for any $\e>0$ we can find a continuous function $\tilde \eta:\KK\times\KK\to L^2(J_T,L^2)$ whose image is contained in a finite-dimensional subspace of $C^\infty(J_T,L^2)$ such that 
$$
\sup_{u_0,\hat u\in\KK}\|\RR_T(u_0,h+\tilde \eta(u_0,\hat u))-\hat u\|<\e. 
$$
The required mapping $\varPsi:\KK\times\KK\to L^2(J_T,E_N)$ can now be constructed by repeating literally the argument used in Section~\ref{s2.5}. 

\section{Appendix} 
\label{s4}

\subsection{Approximation of functions valued in a Hilbert space}
The following simple result implies, in particular, that when dealing with the property of uniform approximate controllability, one can always assume that the image of the corresponding control operator lies in a finite-dimensional subspace. 

\begin{proposition} \label{p4.2}
Let~$\CC$ be a compact metric space, let $H$ be a separable Hilbert space, and let $\varPsi:\CC\to H$ be a continuous mapping. Then, for any dense subspace $H_0\subset H$ and any $\delta>0$, there is a finite-dimensional subspace $H_\delta\subset H_0$ and a continuous function $\varPsi_\delta:\CC\to H$ whose image is contained in~$H_\delta$ such that
\begin{equation} \label{4.11}
\sup_{y\in\CC}\|\varPsi(y)-\varPsi_\delta(y)\|_H<\delta. 
\end{equation}
\end{proposition}

\begin{proof}
Let $H^n$ be an increasing sequence of finite-dimensional subspaces such that $\cup_nH^n$ is dense in~$H_0$ and, hence, in~$H$. We denote by~$P_n$ the orthogonal projections in~$H$ onto the subspace~$H^n$. Then the sequence~$\{P_n\}$ converges to the identity  in the strong operator topology. It is well known that, in this case, $P_nu\to u$ as $n\to\infty$ uniformly with respect to~$u$ varying in a compact subset of~$H$. It follows that 
$$
\sup_{y\in\CC}\|\varPsi(y)-P_n\varPsi(y)\|_H\to0\quad\mbox{as $n\to\infty$}. 
$$
We see that, for any $\delta>0$ and a sufficiently large integer $n=n(\delta)$, the function $\varPsi_{\delta}(y)=P_{n(\delta)}\varPsi(y)$ satisfies the required property. 
\end{proof}

\subsection{Approximation by piecewise constant functions}
Let us fix $T>0$. For given integers $s\ge1$ and $r\in[1,s]$, we denote $t_r=rT/s$ and write~$I_{r,s}(t)$ for the indicator function of the interval $[t_{r-1},t_r)$. The following proposition shows that one can approximate square-integrable functions depending on a parameter by piecewise constant functions of a special form. 

\begin{proposition} \label{p4.1}
Let~$\CC$ be a compact metric space, let $G$ be a $d$-dimensional vector space, and let $\eta:\CC\to L^2(J_T,G)$ be a continuous function. Then for any basis $e_1,\dots,e_d$ of~$G$ the function~$\eta$ can be approximated, within any accuracy, by functions of the form
\begin{equation} \label{4.1}
\zeta(y)=\sum_{r=1}^{s}\sum_{l=1}^{2d}c_{lr}(y)I_{r,s}(t)\eta^l,
\end{equation}
where $c_{lr}:\CC\to \R$ are non-negative continuous functions such that
\begin{equation} \label{4.2}
\sum_{l=1}^{2d} c_{lr}(y)\equiv1\quad\mbox{for any $r=1,\dots,s$}, 
\end{equation}
$\eta^l=Ce_l$ for $1\le l\le d$, $\eta^l=-Ce_{l-d}$ for $d+1\le l\le 2d$, and $C>0$ is a number. 
\end{proposition}

\begin{proof}
We wish to prove that, for any $\e>0$, there is a function $\zeta:\CC\to L^2(J_T,G)$ of the form~\eqref{4.1} such that 
\begin{equation*} \label{4.3}
\sup_{y\in\CC}\|\eta(y)-\zeta(y)\|_{L^2(J_T,G)}<\e. 
\end{equation*}
In view of Proposition~\ref{p4.2}, since $C(J_T,G)$ is dense in $L^2(J_T,G)$, there is no loos of generality in assuming that~$\eta$ is a continuous function from~$\CC$ to a finite-dimensional subspace of $C(J_T,G)$. 

Let us introduce a scalar product~$(\cdot,\cdot)$ in~$G$ for which $\{e_l\}$ is an orthonormal basis. Then~$\eta$ can be written in the form
\begin{equation} \label{4.4}
\eta(y)=\eta(y;t)=\sum_{l=1}^d\varphi_l(y;t)e_l,
\end{equation}
where $\varphi_l(y;t)=(\eta(y;t),e_l)$. Note that $\varphi_l$ is a real-valued continuous function on~$\CC\times J_T$. Let us set 
$$
M=\max_{l,y,t}|\varphi_l(y;t)|,\quad C=Md,
$$
where the maximum is taken over $l=1,\dots,d$ and $(y;t)\in\CC\times J_T$. Then~\eqref{4.4} can be rewritten as 
$$
\eta(y;t)=\sum_{l=1}^d\frac{\varphi_l(y;t)+M}{2C}\,\eta^l
+\sum_{l=1}^d\frac{M-\varphi_l(y;t)}{2C}\,\eta^{l+d}
=\sum_{l=1}^{2d}\psi_l(y;t)\eta^l,
$$
where $\psi_l:\CC\times J_T\to\R$ are non-negative continuous functions whose sum is identically equal to~$1$. It remains to note that~$\psi_l$ can be approximated, within any accuracy, by piecewise constant functions of the form $\sum_r c_r(y)I_{r,s}(t)$. 
\end{proof}

\addcontentsline{toc}{section}{Bibliography}
\def\cprime{$'$} \def\cprime{$'$} \def\cprime{$'$}
\providecommand{\bysame}{\leavevmode\hbox to3em{\hrulefill}\thinspace}
\providecommand{\MR}{\relax\ifhmode\unskip\space\fi MR }
\providecommand{\MRhref}[2]{%
  \href{http://www.ams.org/mathscinet-getitem?mr=#1}{#2}
}
\providecommand{\href}[2]{#2}

\end{document}